\documentclass[a4paper]{article}

\newcommand{\COLORON}{1}
\newcommand{\NOTESON}{0}
\newcommand{\Debug}{0} 

\usepackage[usenames]{color} 
\usepackage{amsthm,amssymb,amsmath,bbm,enumerate,graphicx,epsf,stmaryrd,accents,ulem}
\usepackage[bookmarks, colorlinks=false, breaklinks=true]{hyperref} 

\usepackage{authblk}

\newcommand{\comment}[1]{}
\newcommand{\COMMENT}[1]{}

\definecolor{darkgray}{rgb}{0.3,0.3,0.3}
\newcommand{\defi}[1]{{\color{darkgray}\emph{#1}}}

\newcommand{\acknowledgement}{\section*{Acknowledgement}}

\comment{
	\begin{lemma}\label{}	
\end{lemma}
\begin{proof} 

\end{proof}

\begin{theorem}\label{}
\end{theorem} 
\begin{proof} 	

\end{proof}

}

\newtheorem{proposition}{Proposition}[section]
\newtheorem{definition}[proposition]{Definition}
\newtheorem{theorem}[proposition]{Theorem}
\newtheorem{corollary}[proposition]{Corollary}

\newtheorem{lemma}[proposition]{Lemma}
\newtheorem{observation}[proposition]{Observation}
\newtheorem{conjecture}{{Conjecture}}[section]

\newtheorem{problem}[conjecture]{{Problem}}

\theoremstyle{definition}
\newtheorem{remark}{Remark}[subsection]

\theoremstyle{definition}
\newtheorem{examp}[proposition]{Example}

\newcommand{\FIG}{0}

\ifnum \NOTESON = 1 \newcommand{\note}[1]{ 

\hspace*{-30pt}
	{\color{blue}  NOTE: \color{Turquoise}{\small  \tt \begin{minipage}[c]{1.1\textwidth}  #1 \end{minipage} \ignorespacesafterend }} 
	
	}
\else \newcommand{\note}[1]{} \fi

\newcommand{\afsubm}[1]{ \ifnum \Debug = 1 {\mymargin{#1}}
\fi} 

\ifnum \Debug = 1 
\else  \fi

\ifnum \FIG = 1 \newcommand{\fig}[1]{Figure ``{#1}''}
\else \newcommand{\fig}[1]{Figure~\ref{#1}} \fi

\ifnum \FIG = 1 
\else  \fi

\ifnum \Debug = 1 \usepackage[notref,notcite]{showkeys}
\fi

\ifnum \COLORON = 0 \renewcommand{\color}[1]{}
\fi

\newcommand{\pdfFig}[3]{
   \begin{figure}[htbp]
   \centering
   \noindent
   \includegraphics[width=#3\linewidth]{#1.pdf}
   \caption{\small #2}
   \label{#1}
   \end{figure}
}

\newcommand{\N}{\ensuremath{\mathbb N}}
\newcommand{\R}{\ensuremath{\mathbb R}}

\newcommand{\Z}{\ensuremath{\mathbb Z}}

\newcommand{\BS}{\ensuremath{\mathbb S}}

\newcommand{\cc}{\ensuremath{\mathcal C}}

\newcommand{\cf}{\ensuremath{\mathcal F}}

\newcommand{\cs}{\ensuremath{\mathcal S}}
\newcommand{\cgr}{\ensuremath{\mathcal R}}
\newcommand{\ct}{\ensuremath{\mathcal T}}
\newcommand{\cu}{\ensuremath{\mathcal U}}
\newcommand{\cv}{\ensuremath{\mathcal V}}

\newcommand{\sm}{\backslash}

\makeatletter
\DeclareRobustCommand{\cev}[1]{%
  \mathpalette\do@cev{#1}%
}
\newcommand{\do@cev}[2]{%
  \fix@cev{#1}{+}%
  \reflectbox{$\m@th#1\vec{\reflectbox{$\fix@cev{#1}{-}\m@th#1#2\fix@cev{#1}{+}$}}$}%
  \fix@cev{#1}{-}%
}
\newcommand{\fix@cev}[2]{%
  \ifx#1\displaystyle
    \mkern#23mu
  \else
    \ifx#1\textstyle
      \mkern#23mu
    \else
      \ifx#1\scriptstyle
        \mkern#22mu
      \else
        \mkern#22mu
      \fi
    \fi
  \fi
}

\makeatother

\newcommand{\nin}{\ensuremath{{n\in\N}}}

\newcommand{\pth}[2]{\ensuremath{#1}\text{--}\ensuremath{#2}~path}

\newcommand{\seq}[1]{\ensuremath{(#1_n)_{n\in\N}}} 
 
\newcommand{\g}{\ensuremath{G\ }}
\newcommand{\G}{\ensuremath{G}}

\newcommand{\ceil}[1]{\ensuremath{\left\lceil #1 \right\rceil}}
\newcommand{\floor}[1]{\ensuremath{\left\lfloor #1 \right\rfloor}}

\newcommand{\gve}{\ensuremath{G=(V,E)}}

\newcommand{\Cg}{Cayley graph}

\newcommand{\Lr}[1]{Lemma~\ref{#1}}

\newcommand{\Tr}[1]{Theorem~\ref{#1}}

\newcommand{\Sr}[1]{Section~\ref{#1}}

\newcommand{\Prr}[1]{Pro\-position~\ref{#1}}
\newcommand{\Prb}[1]{Problem~\ref{#1}}
\newcommand{\Cr}[1]{Corollary~\ref{#1}}
\newcommand{\Cnr}[1]{Con\-jecture~\ref{#1}}
\newcommand{\Or}[1]{Observation~\ref{#1}}
\newcommand{\Er}[1]{Example~\ref{#1}}
\newcommand{\Dr}[1]{De\-fi\-nition~\ref{#1}}

\newcommand{\impl}[2]{\ref{#1} $\to$ \ref{#2}}

\newcommand{\lf}{locally finite}

\renewcommand{\iff}{if and only if}
\newcommand{\fe}{for every}
\newcommand{\Fe}{For every}

\newcommand{\st}{such that}

\newcommand{\ti}{there is}
\newcommand{\ta}{there are}

\newcommand{\wrt}{with respect to}

\newcommand{\Tfae}{The following are equivalent}

\newcommand{\labtequ}[2]{
 \begin{equation} \label{#1} 	\begin{minipage}[c]{0.9\textwidth}  #2 \end{minipage} \ignorespacesafterend \end{equation} }

\newcommand{\mymargin}[1]{
 \ifnum \Debug = 1
  \marginpar{%
    \begin{minipage}{\marginparwidth}\small%
      \begin{flushleft}%
        {\color{blue}#1}%
      \end{flushleft}%
   \end{minipage}%
  }%
 \fi
}%

\newcommand{\extras}[1]{
 \ifnum \Debug = 1
\section{Extras} #1
 \fi
}%

\newcommand{\mySection}[2]{}

\renewcommand{\defi}[1]{{\color{darkgray}\emph{#1}}}

\usepackage{enumerate,enumitem}

\usepackage[numbers,sort&compress]{natbib}
\usepackage{authblk}
\usepackage[percent]{overpic}
\usepackage{hyperref}

\usepackage{tikz}
\usepackage{geometry}
\usetikzlibrary{arrows.meta, positioning, shapes.geometric}
\usetikzlibrary{positioning, calc, shapes.misc, fit}
\usepackage[nameinlink]{cleveref}

\newcommand{\asly}{asymptotically} 
\newcommand{\asc}{asymptotic} 

\newcommand{\asham}{\asly\ hamiltonian}
\newcommand{\astf}{\asc\ spanning tree family}
\newcommand{\aST}{\asc\ spanning tree}
\newcommand{\infly}{infinitesimally}

\newcommand{\kcon}{$k$-connected}
\newcommand{\gseq}{\ensuremath{(G_n)}}

\newcommand{\sexp}[1]{$#1$-expander}
\newcommand{\sCe}{\sexp{C}}

\newcommand{\diam}{\mathrm{diam}}

\newcommand{\qy}{quasi-isometry}
\newcommand{\qic}{quasi-isometric}
\newcommand{\maqy}{$(M,A)$-quasi-isometry}
\newcommand{\maqic}{$(M,A)$-quasi-isometric}

\newcommand{\hamfam}{hamiltonian family}

\begin{document}

\title{Asymptotic connectivity, hamiltonicity and toughness}

\author{Agelos Georgakopoulos\thanks{Supported by EPSRC grant  EP/V009044/1. I thank the department of mathematics of the NKUA for hospitality
while some of this work was carried out.}}

\affil{ {Mathematics Institute}\\ {University of Warwick}\\  {CV4 7AL, UK}}

\date{\today}
\maketitle

\begin{abstract}
We define coarse variants of the notions of connectivity, hamiltonicity, toughness, and spanning trees. Our notions are invariant under quasi-isometries, and they obey the same relationships between them as their classical analogs. We pose several open problems that are coarse versions of classical results and conjectures  about hamiltonicity.
\end{abstract}



\maketitle

\normalem

\section{Introduction}

There is widespread current interest in studying metric variants of graph properties that are stable under local perturbations and capture the large-scale geometric structure. This parallels ideas in  Geometric Group Theory, much of which is concerned with properties of groups that are invariant under quasi-isometries (of their \Cg s). 

This motivated a quest, urged explicitly by Papasoglu and the author \cite{GeoPapMin}, to seek metric variants of the classical graph-theoretic results and concepts, whereby invariance under quasi-isometry is the key benchmark of a successful variant. See e.g.\ \cite{ADGSmallCounterexamples,AlbGeoBlo,AJKW,BerSeyBou,BlPiPrCoa,BBEGLPS,BLPPSep,ChaFou,DGHLMGallai,FujPapCoa,FujPapAsy,GeoMolQua,HatPilChubby,LiuWeakCoarseMenger,McMFat,NgScSeAsyI,NgScSeAsyII} for a sample of the fast-growing literature in this direction commonly called `Coarse Graph Theory'.

Despite the growing interest, hardly any of the classical concepts of graph theory have so far found coarse variants. This paper partly fills this gap, by providing `asymptotic' variants of the notions of connectivity, hamiltonicity, toughness, and spanning trees. Our asymptotic variants obey the same numerical relationships between them as their classical analogs: asymptotic connectivity is at least twice the asymptotic toughness, and asymptotic hamiltonicity implies that the asymptotic connectivity is at least 2, and the asymptotic toughness at least 1. More importantly, these concepts are all invariant under quasi-isometries. As we will see, the best known results and open problems about hamiltonicity have interesting asymptotic variants.

\subsection{Connectedness} \label{coa con}

Let $G=(V,E)$ be a connected graph, and let $d_G$ denote its graph-distance. The \defi{diameter $\diam_G(U)$} of $U\subseteq V$ is $\sup_{u,v\in U} d_G(u,v)$. We say that a vertex set (usually a separator) $S \subseteq V$ is \defi{$(n,d)$-bounded}, if $S$ is the union of at most $n$ sets each of diameter at most~$d$. We say that \g is \defi{$(k,d,D)$-inseparable}, if for every $(k-1,d)$-bounded set $S \subseteq V$, all but at most one components of $G-S$ lie inside the ball $B_D(S):=\{v \in V \mid d_G(v,S)\leq D\}$.\footnote{Equivalently, at most one component  of $G-S$ has diameter greater than $D'$; see \Lr{lem kdb}.} 


\begin{definition} \label{def asly kcon}
A sequence \seq{G} of  graphs is \defi{\asly\ \kcon}, if \fe\ $d\in \N$ \ti\ $D \in \N$ \st\ $G_n$ is  $(k,d,D)$-inseparable \fe\ $n$.
\end{definition} 


This definition can be applied with \gseq\ replaced by a single, infinite, graph \G\ (by adapting the definition, or by just setting $G_n=G$ \fe\ $n$).\footnote{There is a related notion of \defi{coarse separation}, see e.g.\ \cite{PapQI,BGTCoarse,PatilFinitary}. In that terminology, \g is \asly\ \kcon\
\iff\ it \ti\ no $d$ \st\ $G$ is coarsely separated by a family of $(k-1,d)$-bounded separators.} For $k=2$  similar definitions were used in \cite{AlbGeoBlo}.

\smallskip
The \defi{\asc\ connectivity $\kappa_\infty((G_n))$} of \gseq\ is defined as
$$\kappa_\infty((G_n)):=\sup\{k \mid \g \text{ is \asly\ \kcon}\}.$$

We say that two sequences of graphs \seq{G}, \seq{H} are \defi{$(M,A)$-quasi-isometric}, if each member of one sequence is $(M,A)$-quasi-isometric to a member of the other (as defined in \Sr{sec QI}). We will prove, in \Sr{sec QII}, that $\kappa_\infty$ is invariant under this relation: 

\begin{proposition} \label{aslykcon}
Being \asly\ \kcon\ is invariant under quasi-isometry  \fe\ $k\in \N$. (Hence so is $\kappa_\infty$.)
\end{proposition}

Let us consider some examples. Letting $C_n$ denote the $n$-vertex cycle, it is straightforward to check that $\kappa_\infty((C_n))=2$. If \gseq\ is $(M,A)$-quasi-isometric to a family of trees of unbounded diameters, then it easily follows from Manning's bottleneck property \cite{Manning,GeoPapMin} that $\kappa_\infty((C_n))=1$. Similarly, it is easy to see that  every infinite, locally finite, \asly\ 2-connected graph has at most 1 end (\Prr{onended}). 

\begin{examp} \label{ex Zd}
Let $G_n$ denote the $n \times n$ planar square grid (or its toroidal variant). Then $\kappa_\infty((G_n))=\infty$, and similarly, letting $\Z^d, d\geq 2$ denote the $d$-dimensional cubic lattice, we have $\kappa_\infty(\Z^d)=\infty$. This is not hard to see directly, but it will also follow from \Prr{kinfty} below.\footnote{Consequently, high asymptotic connectivity does not force large asymptotic minors in the sense of \cite{GeoPapMin}.} 
\end{examp}


Using this example as a building block, we can construct sequences \seq{G}\ with $\kappa_\infty((G))=k$ \fe\ $k\in \N$:

\begin{examp} \label{ex kcon}
Fix a finite graph \g with edge-connectivity $\lambda(G)=k$. We will obtain a sequence \seq{G}\ by `blowing up' each vertex of \g into a portion of a grid, the size of which grows with $n$, so that $\kappa_\infty((G))=k$. 

For this, let $O_n$ be the ball of radius $n$ around the origin of the infinite triangular grid. 
\Fe\  every vertex $v$ of \G, let $O^v_n$ be a copy of $O_n$. Let $m:=d(v)$ be the degree of $v$, and pick vertices $w^v_1, \ldots, w^v_{m}$ on the `boundary' of $O^v_n$ that are pairwise at distance at least $\floor{\frac{n}{10 m}}$. 
Fix a bijection $\pi: E(v) \to \{1,\ldots,d(v)\}$ for each $v\in V(G)$. For every edge $uv\in E(G)$, identify $w^v_{\pi(e)}$ with  $w^u_{\pi(e)}$, and let $G_n$ be the resulting graph. 

It is easy to see that each cut $\{u_1 v_1, \ldots u_k v_k\}$ of \g consisting of $k$ edges gives rise to a $(k,0)$-bounded separator $S$ of $G_n$, consisting of the `attachment points' of the form $w^{u_i}_{\pi(u_i v_i)}$, separating vertices that are $n$-far from $S$, namely the centers of some $O^v_n$'s. Thus $\kappa_\infty((G_n))\leq k$. On the other hand, it is easy to see that $\kappa_\infty((O_n))=\infty$, and it follows that  $\kappa_\infty((G_n))\geq k$.
\end{examp}

\subsection{Hamiltonicity} \label{coa ham}

Let $G=(V,E)$ be a graph, and $C \subseteq G$ a cycle (or an arbitrary subset of $V$). We say that $C$ is \defi{$R$-dense} in \G, if $d_G(x,C)\leq R$ holds \fe\ $v\in V$. 
This notion generalises to closed curves in metric spaces, and offers naive ways to try to extend hamiltonicity to such spaces. For example, we can call a family \cf\ of metric spaces \defi{weakly hamiltonian}, if there is $R\in \R$ \st\ each $F \in \cf$ admits an $R$-dense simple closed curve. A serious drawback of this notion from the coarse-graph-theoretic perspective however, is that it is not invariant under quasi-isometries (see \Er{}). But we will find a better notion which is. A key insight is that a Hamilton cycle in a graph \g is a closed curve which is bijective when restricted to $V(G)$, and while $R$-denseness is a coarse analogue of \emph{surjectivity}, we also need a coarse analogue of \emph{injectivity}. Let me define this next.

Let $W=x_0 x_1 \ldots x_k=x_0$ be a cycle, or more generally a closed walk in a graph \G.  We say that $W$ is \defi{$(d,D)$-injective}, if for every $i,j$ with $d_G(x_i,x_j)\leq d$,  
at least one of the two subwalks $W'$ of $W$ between  $x_i,x_j$ satisfies $\diam_G(W')\leq D$. 

\begin{definition} \label{def HC}
A sequence \seq{G} of finite graphs is  \defi{\asly\ hamiltonian}, if \fe\ $d \in \N$ there are $D,R \in \N$, and a family of cycles $C_n=C_n^d\subset G_n$ \st\ \fe\ $n$, 
\begin{enumerate}
\item \label{hc i} $C_n$ is $R$-dense in $G_n$, and
\item  \label{hc ii} $C_n$ is (d,D)-injective.
\end{enumerate}
\end{definition}

We remark that for each fixed $n$, only finitely many values of $d$ matter, namely those with $d< \diam(G_n)$; we can fix $C_n^d$ for larger $d$. But for each $d$, all but finitely many $G_n$s matter. We call a family of cycles $(C_n^d), d,n\in \N$ as above a \defi{\hamfam} of \gseq. We can think of $D,R$ as functions of $d$, called the \defi{control functions} of the \hamfam.

%
%

The following is a straightforward consequence of the definitions:
\begin{observation} \label{obs ham con}
If \gseq\ is \asham, then it is \asly\ 2-connected.
\end{observation}

Similarly to \Prr{aslykcon}, we will prove that \asc\ hamiltonicity is also invariant:

\begin{theorem} \label{asham qi}
Let \seq{G}, \seq{H} be quasi-isometric sequences of graphs. Then \seq{H} is \asly\ hamiltonian \iff\  \seq{G} is. 
\end{theorem} 

Here and throughout the paper, we allow a single vertex as a \defi{trivial cycle}, in order to retain the validity of  \Tr{asham qi} in the situation where one of \seq{G}, \seq{H} has a subsequence of bounded diameters. 

It is natural to consider variants of \Dr{def HC} where $R$ is a constant or $C_n^d$ is independent of $d$ \fe\ $n$, but we will prove that these conditions imply that \gseq\ is quasi-isometric to a sequence of cycles (\Tr{thm quacy}). But \Dr{def HC} allows for a much wider variety of \asly\ hamiltonian sequences: 

\begin{examp}  \label{ex Hilbert}
Let $G_n$ be the $n \times n$ square grid, and $X_n$ the euclidean square of side length $n$. Then both $(G_n), (X_n)$ are \asly\ hamiltonian\footnote{We have only formally defined \asc\  hamiltonicity for graphs, but extending it to metric spaces is straightfoward: the notion of $R$-denseness applies verbatim, while the notion of $(d,D)$-injectivity can be extended by replacing `closed walk' by `closed curve'.}. To see this, let $H_k$ denote the $k$th iteration of the construction of the Hilbert curve. A \hamfam\ $(C_n^d)$ for $(X_n)$ can be obtained by scaling  $H_{n-m+1}$ by a factor of $n$, where $m:= \min\{d,n\}$, and its {control functions} $D,R$ are linear in $d$. By an appropriate approximation we obtain a \hamfam\ of \gseq. 
\end{examp}

\subsection{Toughness}

Recall that a graph \G\ is called $t$-tough for some $t \in \R_{>0}$, if for every separator $S\subset V(G)$,  $G-S$ has at most $|S|/t$ components. The following is a coarse version:

\begin{definition} \label{def tough}
We say that a 
\gseq\ is \defi{\asly\ $t$-tough}, if for every $d\in \R_+$ \ti\ $D\in \R_+$ \st\ for every $k\in \N$, and every $(k,d)$-bounded $S \subseteq V$, all but at most $k/t$  components of $G-S$ lie inside $B_D(S)$. 
\end{definition}

Let ${\tau_\infty(G)}:= \sup\{ t \mid \g \text{ is \asly\ $t$-tough}\}$ be the \defi{\asc\ toughness} of \G. 

Again this is a \qy\ invariant: 
\begin{proposition} \label{obs tou}
If $G,H$ are \qic, then $\tau_\infty(G)=\tau_\infty(H)$.
\end{proposition}

As in the non-asymptotic case, an immediate consequence of the definitions is that 
\begin{observation} \label{obs tou k}
$k_\infty(G) \geq 2\tau_\infty(G)$.
\end{observation}

It is a well-known exercise that every hamiltonian graph is (at least)  1-tough. We will prove the coarse version:
\begin{proposition} \label{tough 1}
If \seq{G}\ is \asly\ hamiltonian then it is \asly\ 1-tough.
\end{proposition}

\subsection{The main problems} \label{sec main OP}

Some of the best-known results and open problems about hamiltonicity give rise to intriguing problems on our asymptotic analogs. Be begin with a famous conjecture about \Cg s: 
\begin{conjecture}[folklore \cite{BaAut}]  \label{conj HCG}
Every finite connected Cayley graph on at least 3 vertices is hamiltonian.
\end{conjecture}

We provide two coarse versions of this. 
Let \defi{\cc} be the family of all finite connected \Cg s. 
\begin{problem}  \label{prob groups weak}
Is \cc\ weakly hamiltonian? 
\end{problem} 

\begin{problem}  \label{prob groups}
Is \cc\ \asly\ hamiltonian? 
\end{problem} 

If the answers to these 
problems are positive, one can further ask for the smallest constant $R$ making \Prb{prob groups weak} true, and the slowest growing control functions $D,R$ making \Prb{prob groups} true. Moreover, the same problems can be posed for the family \ct\ of finite vertex-transitive graphs as well. 
\medskip

Let me argue that \Prb{prob groups} may be more fruitful than \Cnr{conj HCG}.
Much of the research on the latter needs strong assumptions not only on the group, but also on the generating set defining the \Cg, see \cite{CURRAN1996,KutMarVT} for surveys. \Tr{asham qi} suggests that the choice of generating set matters less in \Prb{prob groups}. 

In support of \Prb{prob groups}, we will prove, in \Sr{sec cart}, that if $(G_n)$ and $(H_n)$ are asymptotically Hamiltonian sequences of finite (not necessarily Cayley) graphs, then the sequence of their Cartesian products $(G_n\square H_n)$ is also \asham. This suggests that \Prb{prob groups} may be more amenable to recursive arguments than \Cnr{conj HCG}. Indeed, the following analogue of the aforementioned result is wide open. Call a group $\Gamma$ \defi{hamiltonian} if $|\Gamma|\leq 2$, or all connected \Cg s of $\Gamma$ are hamiltonian:

\begin{problem}  \label{prob cart}
Is it true that if a group $\Gamma$ is the cartesian product of two finite hamiltonian groups, then $\Gamma$ must itself be hamiltonian?
\end{problem} 

On the other hand, \Prb{prob groups} can be harder than \Cnr{conj HCG} for specific groups. Let $Q_n$ stand for the $n$-dimensional Hamming hypercube. It is well-known and easy to see that $Q_n$ is hamiltonian, but I do not know how to answer the following: 

\begin{problem}  \label{prob Hamm}
Is $(Q_n)$ \asly\ hamiltonian?
\end{problem} 

As a first step towards this it could be helpful to consider the following problem, which can be thought of as a variant of the snake-in-the-box problem \cite{ZemSna} and may be of independent interest:

\begin{problem}  \label{prob snake}
Is there a universal constant $R$, such that $Q_n$ has an $R$-dense induced cycle for every $n>1$?
\end{problem}

\subsubsection*{Toughness} 
Chv\'atal's Toughness Conjecture \cite{ChvTou} states that there exists $t \in \R_{>0}$ such that every $t$-tough graph has a Hamilton cycle. The following is a coarse version:
\begin{problem} 
Is there $t \in \R_{>0}$ such that every \asly\ $t$-tough 
family of finite graphs is \asly\ hamiltonian?
\end{problem} 

The weak-hamiltonicity variant is also intriguing: 
\begin{problem} \label{prob tough dense}
For which $t \in \R_{>0}$ is there $R\in \N$ \st\ every finite $t$-tough  graph admits an $R$-dense cycle? How does $R$ depend on $t$?
\end{problem} 

We remark the known constructions disproving Chv\'atal's Conjecture for $t< 9/4$ \cite{BauBroVelTough} contain dominating cliques, and therefore 2-dense triangles. I am not aware of any non-trivial lower bound on $t$ for \Prb{prob tough dense}. It has been proved that $R=1$ suffices for $t$ close to 1 under strong additional degree assumptions, see e.g.\ \cite{SaiYamDomCycles}. 

\medskip
\subsubsection*{Expanders}

Given an $n$-vertex graph $G=(V,E)$ and $C\in \R_+$, we say that \g  is a \defi{\sCe}, if 
\begin{enumerate} 
\item \label{ex i} $|N(X)| \geq C|X|$ for every $X \subseteq V$ with $|X| <  n/2C$, and 
\item \label{ex ii} there is an edge between every two disjoint sets of at least $n/2C$ vertices,\footnote{While condition \ref{ex i} is standard throughout the expander literature, \ref{ex ii} is not. Thus the term \sCe\ may be somewhat misleading, but I decided to follow the terminology of \cite{DMMPSHam}.}
\end{enumerate} 
where the neighbourhood \defi{$N(X)$} of $X \subseteq V$ is the set of vertices of $\g \sm X$ sending an edge to $X$.

Dragani\'c, Montgomery, Munh\'a Correia, Pokrovskiy \& Sudakov \cite{DMMPSHam} recently proved that there is a universal constant  $C$ \st\ every \sCe\ is hamiltonian. This deep result implies and extends several hard-earned previous hamiltonicity results, in particular on (pseudo-)random graphs. This motivates:

\begin{problem} \label{prob exp}
Is it true that \fe\ $C\in \R$, the 
family of finite \sCe\ graphs is \asly\ hamiltonian? Is it weakly hamiltonian?
\end{problem} 

As we will see in \Sr{sec exps}, to prove the first question it would suffice to prove it for large enough $C$; this is a consequence of \Tr{asham qi} and the fact that taking powers boosts the expansion. This argument fails however for the second question.

If the answer to the second question of \Prb{prob exp} is positive, it would be interesting to determine the optimal dependence of $R$ on $c$:

\begin{problem} \label{prob R}
Determine the asymptotics of the function 
$$R(c):= \inf\{ R \in \R \cup \{\infty\} \mid \text{ every  $c$-expander has an $R$-dense cycle}\}.$$
\end{problem} 


\subsubsection*{Planar graphs}
Tutte \cite{tutte56} proved that every 4-connected planar graph admits a Hamilton cycle. 

\begin{problem}  \label{Tutte}
Must every \asly\  4-connected sequence of finite planar graphs be \asham?
\end{problem} 

It is well-known that 3-connected planar graphs need not be hamiltonian (in fact the length of the longest cycle can be of order $\Theta(n^{\log_3 2})$) \cite{ChenYu}. We observe that, likewise, \asly\  3-connected planar graphs need not be \asham:

\begin{examp}  \label{ex tcon}
Let $G$ be a finite, planar, cubic, 3-connected graph with no Hamilton cycle, which is well-known to exist. Obtain $G_n$ from \g as in \Er{ex kcon}, and recall that $\seq{G}$ is \asly\ 3-connected. Any $R$-dense cycle $C_n\subset G_n$ must meet the interior of every $O^v_n$ if $n> 2R$. Thus it must intersect 2 of the 3 attachment points of each $O^v_n$, and therefore induce a Hamilton cycle on \G, a contradiction.
\end{examp}

Shan \cite{ShanToughTriangulation} constructed  non-hamiltonian $\frac{3}{2}$-tough plane triangulations.\footnote{The constant $\frac{3}{2}$ here is optimal by Tutte's theorem and the the inequality $\kappa\geq 2 \tau$ (similarly to \Or{obs tou k}).}  This raises:

\begin{problem}  \label{prob shan}
Is there $R\in \N$ \st\ every $\frac{3}{2}$-tough plane triangulation has an $R$-dense cycle? 
\end{problem} 



\subsection{Some fractal graph theory} \label{sec main OP}

The above asymptotic notions admit infinitesimal variants applicable to arbitrary metric spaces of bounded diameter. We can still consider sequences thereof, but our notions now make sense for a single such space $(X,d)$. We just adapt Definitions~\ref{def asly kcon}, \ref{def HC} and \ref{def tough} to require that $\lim_{d\to 0^+} D(d)=0$ and ---for \Dr{def HC}--- also $\lim_{d\to 0^+} R(d)=0$, thereby defining \defi{infinitesimal} connectedness, hamiltonicity, and toughness.\footnote{There are other examples of infinitesimal versions of graph-theoretic results; see \cite[\S~1.5]{BBEGLPS}, \cite[Theorem 8.1.]{GeoPapMin}.} 

For example, it is not hard to see that the Sierpinski gasket \cs, endowed with the euclidean metric, is infinitesimally hamiltonian, using a Hamilton cycle of the graph $S_n$ of the $n$th iteration of the construction of \cs\ \fe\ $n$, as defined in \cite{TegGodSierpinski}, to obtain a \hamfam\ of \cs. Similarly, \Er{ex Hilbert} can be adapted ---in fact simplified--- to prove that the unit square is infinitesimally hamiltonian.

\pdfFig{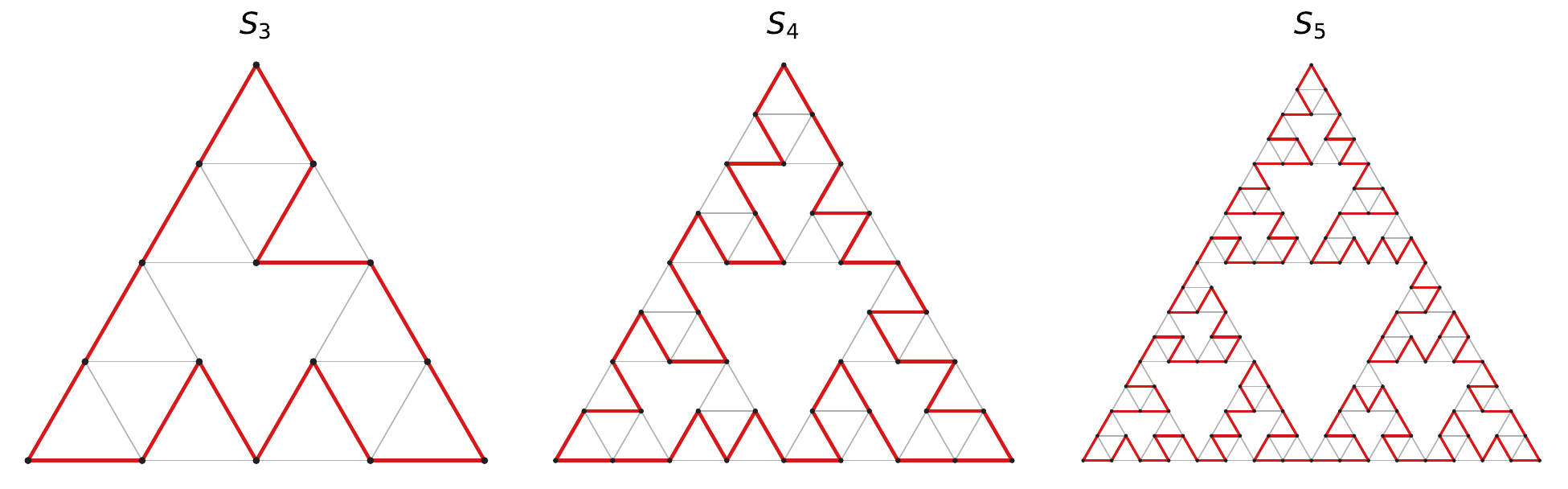}{Initial members of a \hamfam\ of the Sierpinski gasket.}{.8}

Many of our above statements and their proofs apply almost verbatim to the infinitesimal variants, and this applies to some of our problems too. For example, we have an infinitesimal variant of \Prb{Tutte}:   

\begin{problem}  \label{Tutte e}
Must every \infly\  4-connected Riemannian surface that embeds into $\BS^2$ be infinitesimally hamiltonian?
\end{problem}

\subsection{Relation to small generators of the cycle space}  \label{intro lssc}

We say that a graph \g is \defi{$k$-generated} if its cycle space (over $\Z_2$) is generated by a set of cycles each of length at most $k$. We say that \gseq\ is $k$-generated if each $G_n$ is. In \Sr{sec kgen} we will prove:

\begin{proposition} \label{kinfty}
Let \gseq\ be a $k$-generated graph sequence for some $k\in \N$.
If $\gseq$ is asymptotically $2$-connected, then it is asymptotically $n$-connected for every $n\in\mathbb N$. 
\end{proposition}

This complements a result of Albrechtsen \cite{AlbCoa}, saying that $k$-generated graphs satisfy a coarse version of Menger's theorem.

Since grids are 4-generated, \Prr{kinfty} implies the claim we made in \Er{ex Zd}. More generally, it implies that each finitely presented, 1-ended \Cg\  has infinite asymptotic connectivity. But in fact this extends to the non-finitely-presented case and beyond: 

\begin{proposition} \label{onended}
Let \g be a connected, infinite, \lf, vertex-transitive graph. Then $\kappa_\infty(G)=\infty$ \iff\ \g contains infinitely many pairwise disjoint rays\footnote{A \defi{ray} is a 1-way infinite path.} and no two rays of \g can be separated by a finite vertex set.\footnote{In other words, \g is 1-ended, and its unique end is thick.}
\end{proposition}


\subsection{Spanning trees} \label{sec ast}

Our  notions of coarse surjectivity and injectivity used to define \asc\ hamiltonicity can be applied more generally, and in particular we can adapt \Dr{def HC} to define an {\aST} of \gseq. To begin with, we say that a tree $T\subseteq G$ is $(d,D)$-injective in \G, if for every $x,y\in V(T)$ with $d_G(x,y)\leq d$, the \pth{x}{y}\ $P$ in $T$ satisfies $\diam_G(P)\leq D$.  An  \defi{asymptotic spanning tree (family)} of \gseq\ can then be defined as in \Dr{def HC}, except that each $C^d_n$ is a tree rather than a cycle: namely an $R(d)$-dense and $(d,D(d))$-injective subtree of $G_n$. Unlike \Dr{def HC}, this definition can be applied to a singleton infinite graph \G.\footnote{In fact, one can try to extend \Dr{def HC} to (countably) infinite 1-ended graphs using the topological notion of Hamilton circle \cite{DiestelBook25}.}

An easy example of a family \gseq\ that does not have an \astf\ is obtained by letting $G_n$ be a cycle of length $n$ with pendant paths of length $n$ attached to each vertex. On the other hand, it is easy to turn any \asc\ \hamfam\ into an \astf\ (with larger $R$) by removing a subpath of small diameter. I would be interested to see natural sufficient or necessary conditions guaranteeing \aST s. For example, \Prr{kinfty} motivates the following:  

\begin{problem}  \label{}
Let $k\in \N$, and let \gseq\ be an \asly\  2-connected sequence of $k$-generated  graphs. Must \gseq\ have an \aST? Must it be asympto\-tically hamiltonian?
\end{problem} 

The same questions can be asked with the condition of being $k$-generated replaced by a) $\delta$-hyperbolicity in the sense of Gromov \cite{gromov} b) other curvature-type conditions, or c) being a triangulation of $\BS^2$.  For the first question, the condition of being \asly\  2-connected may be superfluous. 

Some of the above problems on \asc\ hamiltonicity, in particular \Prb{prob groups} and \ref{prob Hamm}, have natural relaxations for \aST s, and I do not have proofs for them. But \aST s can raise additional questions about infinite graphs:

\begin{problem}  \label{prb Cgast}
Does every finitely generated \Cg\ have an \aST? 
\end{problem} 

\smallskip

Let me make a side remark about \Cnr{sec main OP}. Motivated by that problem, Benjamini\footnote{private communication in 2023} asked for the following relaxation, which we will prove using toughness:\footnote{This result was independently discovered, with the same proof, by Li \& Methuku \cite{LiMet}.}

\begin{proposition}\label{}
Every finite, connected, vertex-transitive graph \g has a spanning tree of maximum degree at most 3.
\end{proposition} 
\begin{proof} 	
It is known that \g must be 1-tough \cite[p.~25]{BaAut}. By a theorem of Win \cite{WinKTrees} (see \cite{ShaLon} for details), every $t$-tough graph has a spanning tree of maximum degree at most $\ceil{1/t}+2$. 
\end{proof}
\section{Preliminaries}

We follow the terminology of Diestel~\cite{DiestelBook25} for graph-theoretic terms, and that of  Godsil \& Royle \cite{GR01} for the few group theoretic ones that we will need.
Disconnected graphs will be irrelevant.


A graph $G=(V,E)$ is \defi{hamiltonian} if it contains a \defi{Hamilton cycle}, i.e.\ a cycle containing all of $V$.

\subsection{Distances} \label{sec dist}


Let $G$ be a connected graph.
We write~\defi{$d_G(v, u)$} for the distance of the two vertices~$v$ and~$u$ in~$G$, i.e.\ the minimal number of edges in any \pth{v}{u}. 

For two sets~$U$ and~$U'$ of vertices of~$G$, we write~\defi{$d_G(U, U')$} for the minimum distance of two elements of~$U$ and~$U'$, respectively.
If one of~$U$ or~$U'$ is just a singleton, then we omit the braces, writing $d_G(v, U') := d_G(\{v\}, U')$ for $v \in V(G)$.


Given a set~$U$ of vertices of~$G$, the \defi{ball (in~$G$) around~$U$ of radius $r \in \N$}, denoted by~\defi{$B_U(r)$},  is the set of all vertices in~$G$ of distance at most~$r$ from~$U$ in~$G$. If~$U = \{v\}$ for some~$v \in V(G)$, then we omit the braces, writing~$B_v(r)$.

\subsection{Quasi-isometries} \label{sec QI}


\medskip

For $M \in \R_{\geq 1}$ and $A \in \R_{\geq 0}$, an \defi{$(M, A)$-quasi-isometry} from a graph~$G$ to a graph~$H$ is a map~$\varphi : V(G) \rightarrow V(H)$ such that
\begin{enumerate}[label=\rm{(Q\arabic*)}]
    \item \label{quasiisom:1} $M^{-1} \cdot d_G(g,h) - A \leq d_H(\varphi(g),\varphi(h)) \leq M\cdot d_G(g,h)+A$ for every $g,h \in V(G)$, and
    \item \label{quasiisom:2} for every $v \in V(H)$ there is $h \in V(G)$ such that $d_H(v,\varphi(h)) \leq A$.
\end{enumerate}

\medskip

The following is probably well-known to experts on coarse separation, but I could not find an explicit reference. Given a graph $G=(V,E)$ and $S \subseteq V$, we say that $S$ \defi{separates} $x,y\in V$ if $S$ intersects every \pth{x}{y} (i.e.\ we allow $\{x,y\} \cap S \neq \emptyset$). 

\begin{lemma}\label{lem fS D}	
Let $\phi: V(G) \to V(H)$ be a \maqy, and suppose $S \subset V(G)$ separates $x,y\in V(G)$, then  $B_{K}(\phi(S))$  separates  $\phi(x),\phi(y)$ in $H$, where $K=K(M,A)$ depends on $M,A$ only.
\end{lemma}
\begin{proof} 
Let
$P= (\phi(x)=)y_0,y_1,\ldots,y_n(=\phi(y))$
be a $\phi(x)$--$\phi(y)$~path in $H$. For every $i$, choose
$z_i\in V(G)$ such that
$d_H(y_i,\phi(z_i))\leq A$, 
taking $z_0=x$ and $z_n=y$. Since $y_i$ and $y_{i+1}$ are adjacent in $H$, we have 
$d_H(\phi(z_i),\phi(z_{i+1})) \leq 1+2A$. 
This implies 
\begin{equation}
 d_G(z_i,z_{i+1})
 \leq M\bigl(d_H(\phi(z_i),\phi(z_{i+1}))+A\bigr)
 \leq M(1+3A)=:L.
 \label{eq:domain-step}
\end{equation}
For each $i$, choose a shortest $z_i$--$z_{i+1}$ path in $G$.
Concatenating these paths gives an $x$--$y$ walk in $G$, which contains
an $x$--$y$~path. By assumption, this path meets $S$. Therefore, for some
$i$, the chosen $z_i$--$z_{i+1}$ path contains a vertex $s\in S$. By
\eqref{eq:domain-step},
$
 d_G(z_i,s)\leq L.
$
Thus we have
\begin{equation}
\begin{aligned}
 d_H(y_i,\phi(S))
 &\leq d_H(y_i,\phi(z_i))+d_H(\phi(z_i),\phi(s))\\
 &\leq A+M d_G(z_i,s)+A\\
 &\leq ML+2A.
\end{aligned}
\end{equation}
Letting  $K=K(M,A):=ML+2A$, we thus have $y_i\in B_K(\phi(S))$. Since $P$ was arbitrary, every $\phi(x)$--$\phi(y)$ path in $H$ meets $B_K(\phi(S))$ as claimed.
\end{proof}

\section{Coarse invariance} \label{sec QII}

In this section we prove that our notions of asymptotic connectivity, hamiltonicity, and toughness are invariant under quasi-isometries. We start with asymptotic connectivity:

\begin{proof} [Proof of \Prr{aslykcon}]
Let $\phi: V(G) \to V(H)$ be a \maqy.  Suppose \g fails to be \asly\ \kcon, which means that for some $d$, and arbitrarily large $D$,  \ta\  $(k-1,d)$-bounded separators $S_D$ witnessing that \g  is not $(k,d,D)$-inseparable. 
Then by \Lr{lem fS D}, $S'_D:= B_{K(M,A)}\phi(S_D)$ is a family of separators of $H$ \st\ 
each $S'_D$ is $(k-1,d')$-bounded 
for $d':=M(d+A)+ 2K(M,A)$, with at least two components containing vertices that are at least $D':= (M^{-1}D-A) - K(M,A)$ far from $S'_D$. Since $D'\to \infty$ as $D \to \infty$, this proves that $H$ is not $(k,d',D)$-inseparable for arbitrarily large $D$. Thus if \g fails to be \asly\ \kcon, then so does $H$. Since quasi-isometry is an equivalence relation, the converse holds too. 
\end{proof}

The proof of \Prr{obs tou} is similar, except that the parameter $k$ can vary with $D$. I leave the details to the interested reader.

\medskip
We finish this section with the following observation, providing an alternative definition of \asc\ connectivity alluded to in the introduction; namely, instead of requiring that at most one component of $G-S$ is not inside the $D$ neighbourhood of $S$ in the definition of $(k,d,D)$-inseparable, we could require that at most one component  of $G-S$ has diameter greater than $D'$: 
\begin{observation}\label{lem kdb}	
Let $\gve$ be a graph, let  $S$ be a $(k,d)$-bounded subset of $V$, and $C$ a component of $G -S$. If $C \subseteq B_R(S)$, then $\diam_G(C)< D(R):= 2k(2R+d)$.
\end{observation}
\begin{proof} 
Let $S_1, \ldots, S_k$ be a decomposition of $S$ into subsets of diameter at most $d$ (some of which may be empty). Since $C \subseteq B_R(S)$, we can write $C$ as a union of subsets $C_1, \ldots, C_k$ \st\ $d_G(x,S_i)\leq R$ holds \fe\ $x\in C_i, i\in [k]$. To bound $\diam_G(C)$, pick $v,w\in V(C)$, and a sequence $v=x_1, \ldots, x_n=w$ \st\ $d_G(x_j,x_{j+1}) \leq 2R+d$, and minimizing $n$; such a sequence exists since $C$ is connected and any \pth{v}{w}\ in $C$ is a candidate. Note that our sequence cannot meet any $C_i$ more than twice, because any $x_j,x_m\in C_i$ satisfy $d_G(x_j,x_m)\leq 2R+d$ by the triangle inequality. Thus $n\leq 2k$ by the pigeonhole principle. Applying the  triangle inequality to our sequence thus yields $d_G(v,w)\leq 2k(2R+d)$, and so $\diam_G(C)$ satisfies the same bound. 
\end{proof}

\subsection{Coarse invariance of \asc\ hamiltonicity}

In this section we prove \Tr{asham qi}. A \defi{$(d,D,R)$-cycle} is a cycle $C$ in a graph \g satisfying both conditions of \Dr{def HC}, i.e.\ $C$ is $R$-dense and $(d,D)$-injective in $G$. We observe that this definition  extends verbatim to closed \defi{walks}:  
a {walk} in $G$ is a finite sequence
$W=x_0x_1\ldots x_\ell$ of vertices of $G$ such that $x_ix_{i+1}\in E(G)$ for every
$i\in [\ell-1]$. 
We say that $W$ is \defi{closed} if $x_0=x_\ell$.

More generally, if instead of $x_ix_{i+1}\in E(G)$ we require $d_G(x_i,x_{i+1})\leq K$ for some $K\in \N$, then we obtain a (closed) \defi{$M$-near-walk}. The above definition of $(d,D,R)$-cycle extends verbatim to near walks.

\medskip
We first relax \Tr{asham qi} by replacing cycles by (near) walks:

\begin{lemma} \label{lem walks}
Let $W$ be a closed $(d,D,R)$-walk in a graph \G, and let $\phi: G \to H$ be a \maqy. Then $\phi(W)$ is a closed, $(d',D')$-injective, $R'$-dense, $(M+A)$-near-walk in $H$, where $d'= \frac{d}{M}-A, \quad D'= MD+A$, \quad and $R'= MR+2A$.
\end{lemma}

Next, we observe that a quasi-dense quasi-injective closed {\bf walk} implies the existence of a quasi-dense quasi-injective {\bf cycle} with appropriately adapted constants: 

\begin{lemma} \label{lem walk cycle}	
The image of any closed $(d,D,R)$-walk $W$ in a graph \G\ contains a\\ (possibly trivial) $(d,D,R+2D)$-cycle.
\end{lemma}

Before proving these lemmas let us see how they imply the coarse invariance of \asc\ hamiltonicity:

\begin{proof}[Proof of \Tr{asham qi}] 	
Suppose \seq{G}\ is \asly\ hamiltonian, and \fe\ $d\in \N$,  let \seq{C^d}\ be a sequence of cycles as in \Dr{def HC}. Let $\phi_n: G_m \to H_n$ be a \maqy\ \fe\ $n$, and let $W^d_n:= \phi_n(C^d_m)$. We think of $W^d_n$ as a closed walk in $H_n$. By \Lr{lem walks}, $W^d_n$ is a 
closed, $(d',D')$-injective, $R'$-dense, $(M+A)$-near-walk with the given parameters. We can extend $W^d_n$ into a closed (1-near) walk $Z=Z^d_n$ by replacing any two consecutive vertices by a path of length at most $K:=M+A$ joining them. Clearly, $Z$ is still $R'$ dense in $H$, and we claim that it is 
$(d'-K,D'+K)$-injective. Indeed, suppose 
two vertices $x,z$ of $Z$ satisfy $d_H(x,z)\leq d'-K$
and choose corresponding vertices $x',z'$ with
$d_H(x,x'),d_H(z,z') \leq K/2$, which implies  $d_H(x',z')\leq d'$.


Since $W^d_n$ is $(d',D')$-injective, at least one of its two $x'$-$z'$ subwalks has diameter at most $D'$. It follows that the corresponding subwalk of $Z$ has diameter at most $D'+K$. 

Plugging $Z$ into \Lr{lem walk cycle} now yields a $(d'-K,D'+K,R'+2(D'+K))$-cycle $\cc^d_n$ in $H_n$. Note that $d'-K \to \infty$ as $d \to \infty$, and so the sequences \seq{\cc^d}\ prove that \seq{H}\ is \asly\ hamiltonian.
\end{proof}

It remains to prove the above lemmas: 

\begin{proof}[Proof of \Lr{lem walks}]
Clearly, $W':=\phi(W)$ is a closed $(M+A)$-near-walk.
 \smallskip

For quasi-surjectivity, since $W$ is $R$-dense in \G, and $\phi(G)$ is $A$-dense in $H$, it is straightforward to see that $\phi(W)$ is  $R'$-dense in $H$ for $R'=A+(MR+A)= MR+2A$. 
\smallskip

To check quasi-injectivity, let $x,z$ be vertices visited by $W$, and suppose that $x':=\phi(x),z':=\phi(z)$ satisfy $d_H(x',z')\leq d'$. 
Note that
$$d_G(x,z) \leq M(d_H(x',z')+A) \leq M(d'+A) = M((\frac{d}{M}-A)+A) = d.$$
Since $W$ is $(d,D)$-injective, one of the two corresponding $x$-$z$~subwalks $U$ of $W$ satisfies $\diam_G(U)\leq D$. 
It follows that $\diam_H(\phi(U))\leq MD+A=D'$.
Thus $\phi(W)$ is $(d',D')$-injective as desired.
\end{proof}

\begin{proof}[Proof of \Lr{lem walk cycle}]
A \defi{subwalk} of a closed walk $W=x_0 x_1 \ldots x_k=x_0$ is a sequence $x_i x_{i+1\hspace{-1mm}\pmod k} \ldots x_{i+m\hspace{-1mm}\pmod k}$ for some $i,m\leq k$. (In particular $W'$, can consist of a final subsequence of $W$ followed by an initial subsequence.)

Call a closed subwalk $W'$ of $W$ \defi{removable}, if the diameter of its image is at most $2D$. Note that any vertex $x$ occuring $k>1$ times on $W$ separates $W$ into $k$ closed subwalks, which we call the \defi{$x$-subwalks}. We claim that 
\labtequ{removable}{at most one $x$-subwalk is non-removable, while the union of the other $k-1$ $x$-subwalks is removable.}
To see this, suppose there are two $x$-subwalks $U_1 \neq U_2$, each visiting a vertex $u_1,u_2$ respectively \st\ $d(u_i,x)>D$. Then letting $S:=\{x\}$ witnesses that $W$ is not $(0,D)$-injective, contradicting our assumption that $W$ is a $(d,D,R)$-walk. Thus at most one $x$-subwalk contains a vertex $u$ with $d(u,x)>D$, and therefore the union of the other $k-1$ $x$-subwalks has diameter at most $2D$ by the triangle inequality.

A set $F$ of closed subwalks of $W$ is \defi{good}, if their domains are disjoint except perhaps at their endpoints; more precisely, no two $U\neq U'\in F$ share a subsequence of the form $x_i x_{i+1}$. 
Let $\cf$ be the set of good sets of closed subwalks of $W$. We put a partial order $\leq$ on \cf\ by letting $F \leq F'$ if each element of $F$ is a subwalk of some element of $F'$. 

Pick a maximal $F\in \cf$ \wrt\ $\leq$. We claim that 
%
\labtequ{WF cycle}{$W/F$ is injective,}
hence its image is a (possibly trivial) cycle.
Suppose, to the contrary, $W/F$ contains two  occurences of some vertex $y$, and let $Y_1,Y_2$ be the closed subwalks into which these two occurrences separate $W$ (not $W/F$). By \eqref{removable}, at least one of $Y_1,Y_2$ is removable, and we can assume this is $Y_2$. Note that each $U\in F$ is a subwalk of either $Y_1$ or $Y_2$, because the endpoints of $Y_1,Y_2$ appear in $W/F$. Similarly, we cannot have $U=Y_2$ for any $U\in F$. If some $U\in F$ is a subwalk of $Y_2$, then we can obtain a good set $F'$ from $F$ by adding $Y_2$ and removing all elements of $F$ that are subwalks of $Y_2$. But then $F < F'$, contradicting the maximality of $F$. If, on the other hand, each $U\in F$ is a subwalk of $Y_1$, then $F \cup \{Y_2\}$ is a good set, again contradicting the maximality of $F$. This proves \eqref{WF cycle}.

\medskip
Easily, $W/F$ inherits the $(d,D)$-injectivity of $W$. We claim that it is $(R+2D)$-dense in $G$, and to see  this it suffices to check that any vertex $w$ of $W$ which is not in $W/F$ has distance at most $2D$ from  $W/F$. Indeed, such a $w$ lies inside some $U\in F$, and the endpoint $u$ of $U$ lies in $W/F$ since $F$ is good. But $\diam_G(U)\leq 2D$, and so $d(w,u)\leq 2D$ as desired. To summarize,  $W/F$ is a $(d,D,R+2D)$-cycle. 
\end{proof}

\begin{remark} \label{rem gr}
The proof of \Tr{asham qi} yields the stronger statement that the optimal `growth rates' of the control functions $R,D$ coincide for two \qic\ families $\seq{G},\seq{H}$. For example, if \gseq\ admits linear/polynomial/exponential functions $R,D$, then so does $(H_n)$.
\end{remark}

\section{Stronger versions of coarse hamiltonicity and quasi-cycles}


Call a graph sequence \gseq\ \defi{strongly hamiltonian}, if it has a \hamfam\ as in \Dr{def HC} \st\ $R$ is constant ---equivalently, upper bounded. Call \gseq\ \defi{constantly hamiltonian}, if it has a \hamfam\ $(C_n^d)_{n,d\in \N}$ \st\ $C_n^d$ does not depend on $d$, i.e.\ $C_n^{d_1}=C_n^{d_2}$ \fe\ $n,d_1,d_2\in \N$. We will prove that these two conditions are in fact equivalent, and they are equivalent to \gseq\ being a sequence of quasi-cycles: 

\begin{theorem} \label{thm quacy}
\Tfae\ \fe\ sequence \seq{G}\ of graphs: 
\begin{enumerate}
\item \label{qc i} \gseq\ is constantly hamiltonian;
\item \label{qc ii} \gseq\ is strongly hamiltonian;
\item \label{qc iii} There are $M,A$ \st\ each $G_n$  is \maqic\ to a cycle.
\end{enumerate}
\end{theorem} 

It is known that \ref{qc iii} is equivalent to $G_n$ being $(1,A')$-\qic\ to a cycle \cite[\S~4]{GeoPapMin}. 

For the proof of \Tr{thm quacy} we will use the following lemma, which will also be useful in next section.

\begin{lemma}\label{lem intervals}	
Let $C$ be a $(d,D)$-injective cycle in a graph \G. Then $C$ can be partitioned into consecutive, edge-disjoint, subpaths $P_1, \ldots, P_r$, so that $\diam_G(P_i)\leq 3D$, and $d_G(P_i,P_j)>d$ whenever $|i-j|>1\pmod r$. 
\end{lemma}
\begin{proof} 
First, we partition $C$ into consecutive, edge-disjoint, subpaths $Q_1, \ldots, Q_r$, so that $\diam_G(Q_i)\leq D$, but joining any two consequtive ---modulo $r$--- $Q_i's$ results in a path of diameter larger than $D$; such a partition can be easily constructed greedily. Note that 
\labtequ{QiQj}{$d_G(Q_i,Q_j)> d$ holds whenever $|i-j|\geq 3 \hspace{-2mm}\pmod r$,}
because otherwise the $(d,D)$-injectivity of $C$ says that one of the two segments of $C$ between $Q_i,Q_j$ has diameter at most $D$, and in particular there is a pair $Q_{\ell} \neq Q_{\ell+1 \hspace{-1mm} \pmod r}$ the union of which has diameter at most $D$, a contradiction.

\medskip
If $r\leq 3$ we can just let $P_i:=Q_i$ and we are done. Otherwise, if $r$ is even, we join the $Q_i$'s in consecutive pairs to form a new partition $P_1, \ldots, P_{r/2}$ of $C$. If $r$ is odd we do the same, except that we start with $P_1= Q_1 \cup Q_2 \cup Q_3$. Easily, \eqref{QiQj} implies that $d_G(P_i,P_j)>d$ whenever $|i-j|>1\pmod r$. Since  $diam_G(Q_j)\leq D$, we have $diam_G(P_i)\leq 3D$ \fe\ $i$. 
\end{proof}

\begin{proof} 	[Proof of \Tr{thm quacy}]
For the implication \impl{qc iii}{qc i}, let $\phi_n: C_m \to G_n$ be a \maqy, where $C_m$ is a cycle. Plug $\phi_n$ into \Lr{lem walks}, with $W=W_n$ being the `identity' closed walk on $C_m$, noting that $W$ is a $(d,d,0)$-walk \fe\ $d$. Then apply \Lr{lem walk cycle} on $\phi_n(W_n)$ to obtain the desired cycles $C_n\subset G_n$. Since $W_n$ is independent of $d$, so is $C_n$, and it is a \hamfam\ for \gseq\ \wrt\ the functions $D,R$ provided by those lemmas.

\medskip
The implication \impl{qc i}{qc ii} is immediate, since each $C^d_n=C^1_n$ must be $R(1)$-dense.

\medskip
To prove \impl{qc ii}{qc iii} let $(C_n^d)_{n,d\in \N}$ be a \hamfam\ for \gseq\ \wrt\ functions $R,D$, with $R(d) \equiv R$ constant. Let $d_0:= 2R+1$, and $D_0:= D(d_0)$. Given $\nin$, let $C:=C_n^{d_0}$. Thus $C$ is an $R$-dense $(d_0,D_0)$-injective cycle of $G_n$. We apply \Lr{lem intervals} to partition $C$ into  subpaths $P_1, \ldots, P_r$, so that $\diam_{G_n}(P_i)\leq 3D_0$, and $d_{G_n}(P_i,P_j)>d_0$ whenever $|i-j|>1\pmod r$. 
%
%

Let $\cv=V_1, \ldots, V_r$ be a Voronoi decomposition of $G_n$ \wrt\ the $P_i$, i.e.\ each $V_i$ consists of a set of vertices $v\in V(G_n)$ \st\ $d(v,P_i)\leq d(v,P_j)$ \fe\ $j\in [r]$, and $\dot{\bigcup}_{i\in [r]} V_i=V(G_n)$. Let $\phi: V(G_n) \to [r]$ map each vertex $v$ to the index $i$ \st\ $v\in V_i$. We claim that $\phi$ ---which is surjective, as each $v\in P_i$ is mapped to $i$--- is a \qy\ from $G_n$ to the canonical $r$-cycle $K$ on $[r]$. To begin with, note that \fe\ $u\in V(G_n)$ we have 
\labtequ{duPi}{$d_G(u, P_{\phi(u)})\leq R$,}
since $C$ is $R$-dense in $G_n$. 
Therefore, for every $uv\in E(G_n)$, we have\\  $d_G(P_{\phi(u)}, P_{\phi(v)})\leq 2R+1=d_0$ by the triangle inequality. From this we deduce \\ 
$|\phi(u)-\phi(v)|\leq 1 \hspace{-1mm}\pmod r$ by the choice of the $P_i$'s,  
in other words, $d_K(\phi(u),\phi(v)) \leq 1$. Applying this to each edge of a shortest \pth{x}{y}\ for any $x,y\in V(G_n)$ we thus obtain 
$$d_K(\phi(x),\phi(y)) \leq d_G(x,y).$$

Note that the only property of the relation between $d_0$ and $R$ that we are using is that $2R+1\leq d_0$; we will use this in \Cr{cor lin} below.
%
%

On the other hand, if $x,y\in V(G_n)$ satisfy $d_K(\phi(x),\phi(y)) \leq 1$, then $d_G(x,y) \leq 2(R+3D_0)=:M$ holds by \eqref{duPi} and the fact that $\diam_G(P_i)\leq 3D_0$. It follows that 
$$d_G(x,y) \leq M (d_K(\phi(x),\phi(y))+1).$$

These two inequalities mean that $\phi$ is an $(M,M)$-\qy\ assuming $M\geq 1$. \mymargin{this breaks the infinitesimal version}
\end{proof}

The above proof has the following consequence saying that the control function $R$ of any \asham\ \gseq\ must grow at least linearly unless \gseq\ satisfies \ref{qc iii}: 

\begin{corollary} \label{cor lin}
Suppose $(G_n)$  admits a \hamfam\ 
with control functions $D,R$ \st\ $R(d)< cd$ for some constant $c<1/2$. Then $(G_n)$ is quasi-isometric to a sequence of cycles. 
\end{corollary}

\begin{proof}[Proof of \Cr{cor lin}]
As noted in the proof of the implication \impl{qc ii}{qc iii}, the construction of the \qy\ there remains valid as long as $2R+1 \leq d_0$. Thus choosing large enough $d$ that $2R(d) +1 \leq d$ proves our statement. (In fact, it suffices if there is just one value of $d$ \st\ $2R(d)+1 \leq d$.)
\end{proof}

\section{Asymptotic  hamiltonicity implies asymptotic 1-toughness}

As mentioned in the Introduction, our asymptotic notions obey the same numerical relationships between them as their classical analogs; we now prove the last such property remaining:    
\begin{proof}[Proof of \Prr{tough 1}] Suppose not, and fix a $d$ violating 1-toughness. 	
Pick $(d,D,R)$-cycles $C_i\subseteq G_i$, which exist since \gseq\ is \asly\ hamiltonian. Since $d$ violates 1-toughness, we can, for arbitrarily large $D'$, find a separator $S= S_1 \cup \ldots \cup S_k$ in some $G_i$, where $\diam(S_i)\leq d$, \st\ $G_i-S$ has more than $k$ components  
 not contained in $B_{D'}(S)$. We choose $D':= R+D$, and pick $k+1$ of these components $K_1, \ldots, K_{k+1}$, and in each $K_j$ we pick a vertex $x_j$ with $d(x_j, S)>D'$. Since $C_i$ is $R$-dense in $G_i$, we can find $x'_j\in V(C_i)$ \st\ $d(x'_j, x_j) \leq R$ \fe\ $j \in [k+1]$. Note that  $d(x'_j,S)>D'-R=D$ by the triangle inequality. Moreover, $x'_j\in K_j$, because a shortest \pth{x_j}{x'_j}\  avoids $S$ as it has length at most $R$. Thus $S$ separates $x'_j, x'_\ell$ \fe\ $j\neq \ell \in [k+1]$. It follows that $S$ meets all of the $k+1$ subpaths of $C_i \sm \{x'_j \mid j\in [k+1]\}$, and therefore some $S_m, m \in [k]$, must meet at least two of  these subpaths $P_1,P_2$. But then $S_m$ witnesses the fact that $C_i$ is not $(d,D)$-injective, because it interlaces $S_m$ with the endvertices of $P_1$. This contradiction proves that \gseq\ is \asly\ 1-tough.
\end{proof}

\section{Boundedly generated cycle space implies infinite \asc\ connectivity} \label{sec kgen} 

In this section we prove the results announced in \Sr{intro lssc}.

Let $(X,d)$ be a metric space. We say that $S\subseteq X$ is \defi{$M$-near-connected} for $M\in \R_+$, if for every $x,y\in S$, there is a sequence $x=x_0,x_1, \ldots, x_k=y$ of points of $S$ such that $d(x_i,x_{i+1})\leq M$ \fe\ $i<k$. 

For the proof of \Prr{kinfty} we will need the following useful lemma of \cite{TimarCutsets}, which is also implicit in \cite{BaBeCut}: 

\begin{lemma}[{\cite[Theorem~5.1]{TimarCutsets}, \cite[Theorem~1]{BaBeCut}}]
\label{lem:AH-minimal-separator}
Let $G$ be a $k$-generated graph.
Let  $S\subset V(G)$ be an inclusion-minimal set separating $u,v\in V(G)$. Then $S$ is $k/2$-near-connected. 
\end{lemma}

\comment{
Using this, we will prove

\begin{lemma}[{\cite[Theorem~5.1]{TimarCutsets}, \cite[Theorem~1]{BaBeCut}}]
\label{lem n to one}
Let $G$ be a $k$-generated graph, and $u,v\in V(G)$. If $u,v$ can be separated by an $(n-1,d)$-bounded set, then they can be separated by an $(1,d)$-bounded set
\end{lemma}
}

\begin{proof}[Proof of \Prr{kinfty}]
Fix $n\geq 2$ and $d\in\mathbb N$, and let $L:=n(d+k)$. Since $\gseq$ is asymptotically $2$-connected, there exists $D=D(L)$ such that each $G=G_n$ is $(2,L,D)$-inseparable. We claim that $G$ is $(n,d,D)$-inseparable.

Suppose, for a contradiction, that there is a separator
$S=S_1\cup\cdots\cup S_m$, where $m\leq n-1$, that separates vertices $u,v\in V(G)$ both of which are further than $D$ from $S$, and
$\diam_G(S_i)\leq d$ for every $i$.
We may assume that $S$ is inclusion-minimal with the above properties. 
By Lemma~\ref{lem:AH-minimal-separator}, $S$ is  $k/2$-near-connected. It follows that the auxilliary graph $H$ on $[m]$ in which $i$ and $j$ are adjacent whenever
$d_G(S_i,S_j)\leq k/2$, is connected.

We claim that $\diam_G(S)< L$, which contradicts our assumption that $G$ is $(2,L,D)$-inseparable. 

To prove our claim, given $a,b\in S$, choose $i,j\in [m]$ with $a\in S_i$ and $b\in S_j$, and a path 
$i=i_1,\ldots,i_r=j$ in $H$. We have $r\leq m \leq n-1$. By replacing the edges of this path $P$ with paths in \g with length at most $k/2$, and connecting them up using shortest paths (of lengths at most $d$) within each $S_{i_\ell}$, we obtain an \pth{a}{b}\ in \g of length at most $$ (r+1)d+rk/2 \leq nd+(n-1)k/2 < L.$$

We have thus proved that $G$ is $(n,d,D)$-inseparable, and since $n,d$ were arbitrary, that $\kappa_\infty(G)= \infty$.
\end{proof}

Next, we prove that 1-ended vertex-transitive graphs have infinite asymptotic connectivity: 

\begin{proof}[Proof of \Prr{onended}]
For the forward implication, note first that if \g contains a pair of rays $R,L$ that can be separated by a finite vertex set $S$, then $G$ is not $(2,\diam(S),D)$-inseparable for any $D$, as the two components containing $R$ and $L$ have infinite diameters. Thus $\kappa_\infty(G)\leq 1$ in this case. It is well-known \cite[Proposition 5.6]{ThoHad} that if \g does not contain infinitely many pairwise disjoint rays, then it must contain a separator $S$ as above, and so the forward implication is complete. 
\medskip

For the backward implication, suppose \g contains a family \cgr\ of rays as in the statement, but $\kappa_\infty(G) < \infty$. This means that there are $k,d\in \N$ \st\ \g is not  $(k,d,D)$-inseparable for every $D\in \N$. Pick $(k-1,d)$-bounded separators $S_D$ witnessing this. Note that since \g is regular, there is $n< \infty$ \st\ $|S_D| < n$ \fe\ $D$. Since $S_D$ is finite, one of the components of $G - S_D$ meets all elements of \cgr, while at least one other component  $C_D$ contains no infinite subray, while $\diam(C_D)>D$. Pick a vertex $r$ in $C_D$ maximizing $D':=d_G(r,S_D)$. Easily, $D'\to \infty$ as $D\to \infty$. Since \g is vertex-transitive, we may assume (by translating $S_D$) that $r$ is a fixed vertex independent of $D$. Let $D$ be large enough that $B_{D'-1}(r)$ meets at least $n$ rays in \cgr, which exists since $r$ and \cgr\ are fixed and the latter is infinite. We obtain a contradiction, as $S$ must meet all of these disjoint rays, while $|S_D| < n$. 
\end{proof}

\section{Cartesian products} \label{sec cart}

Let $G=(V,E)$ and $H=(V,E)$ be graphs. Their \defi{Cartesian product $G\square H$} is the graph with vertex set $V \times V'$
in which two distinct vertices $(g,h)$ and $(g',h')$ are adjacent if and only if
\begin{equation}
 \bigl(g=g' \text{ and } hh'\in E(H)\bigr)
 \quad\text{or}\quad
 \bigl(h=h' \text{ and } gg'\in E(G)\bigr).
\end{equation}

The aim of this section is to prove the following result, announced in \Sr{sec main OP}: 

\begin{theorem} \label{prod ham}
Let $(G_n)$ and $(H_n)$ be asymptotically Hamiltonian sequences of finite graphs. Then their Cartesian products $(G_n\square H_n)$ form an \asham\ sequence. Moreover, the control functions of $(G_n\square H_n)$ depend linearly on those of $(G_n),(H_n)$.
\end{theorem}

\begin{proof}
Let $(C_n^d)_{n,d\in \N}$ be a \hamfam\ 
 of \gseq, and $(B_n^d)_{n,d\in \N}$ be a \hamfam\ 
 of $(H_n)$, with control functions $R_G,D_G$ and $R_H,D_H$ respectively. We are going to construct a $(d,D',R')$-cycle $K_n^d$ in $X_n:= G_n\square H_n$, which defines a \hamfam\ of $(G_n\square H_n)$.

To do so, given $d\in \N$, apply \Lr{lem intervals} to both $C:= C_n^d$ and $B:= B_n^d$ to decompose them into subpaths $P_1, \ldots, P_r$ and $Q_1, \ldots, Q_k$ respectively, called the \defi{intervals} of $C,B$. Thus we have $\diam_{G_n}(P_i)\leq 3D_G(d)$ and $\diam_{H_n}(Q_i)\leq 3D_H(d)$ \fe\ $i$, while non-consecutive intervals have distance more than $d$. Let $V_C$ denote the set of endpoints of the $P_i$'s, and note that $d_{G_n}(x,y)>d$ \fe\ $x,y\in V_C$ by the latter property. Similarly, let $V_B$ denote the set of endpoints of the $Q_i$'s.

We are going to construct the desired cycle $K_n$ within $C \square B$. Note that there are $|B|$ copies of each $P_i$ in $C \square B$, which we call the fibers of $P_i$. We define the $|C|$ fibers of each $Q_j$ similarly. Form an auxiliary (toroidal) graph $Z$ with vertex set $V(Z) = V_C \square V_B \subseteq V(X_n)$, and edge set corresponding to all above fibers; that is, 
$$E(Z):= \{ (x,y)(z,w) \mid x=z \text{ and $y,w$ are endpoints of the same interval, or vice-versa} \}.$$
Thus $Z$ is isomorphic to the cartesian product of an $r$-cycle and a $k$-cycle. 
In particular, $Z$ is easily seen to contain a Hamilton cycle unless $r+k\leq 3$ (a case that we will handle separately below). 
Note that each cycle in $Z$ gives rise to a cycle in $X_n$, by replacing each edge by the corresponding \defi{interval} of $C \square B$, i.e.\ the fiber of an interval of $C$ or $B$. We call any such cycle a \defi{grid-cycle} of $X_n$. We claim that
\labtequ{toro}{Every grid-cycle cycle $K$ is  $(d,D')$-injective in $X_n$, where\\ $D':=6 ( D_G(d) +  D_H(d))$.}
Indeed, suppose two vertices $(x,y)$ and $(w,z)$ of $K$ satisfy $d_{X_n}((x,y), (w,z))\leq d$. This implies both $d_{G_n}(x,w) \leq d$ and $d_{H_n}(y,z) \leq d$. Therefore, $x,w$ lie on the same or consecutive intervals of $C$ (but cannot be the two endpoints of the same interval), and likewise $x,w$ lie on the same or consecutive intervals of $B$. It follows that $(x,y), (w,z)$ lie on intervals $I,J$ of $C \square B$ that share a vertex or coincide. Thus one of the two subpaths of $K$ between $(x,y)$ and $(w,z)$ is contained in $I\cup J$, and therefore has diameter at most $\diam(I) + \diam(J) \leq 2 (3 D_G(d) + 3 D_H(d))$. This proves \eqref{toro}.
\medskip

If $r=1$, then we can let $K_n^d:= B_n^d$, which is (still) $(d,D_H(d))$-injective, and easily seen to be $(R_H(d)+3D_G(d))$-dense in $X_n$. The case $k=1$ can be handled similarly. 

Thus we can assume that $r,k\geq 2$, and find a Hamilton cycle  $K$ in $Z$. Since each interval has diameter at most $3D_G(d)$ or  $3D_H(d)$,  
we deduce that the corresponding grid-cycle $K_n^d$ is $L$-dense in $C \square B$ for $L:=3D_G(d) + 3D_H(d)$. Since $C$ is $R_G(d)$-dense in $G_n$, and $B$ is $R_H(d)$-dense in $H_n$, it follows similarly that $C \square B$ is $(R_G(d)+R_H(d))$-dense in $X_n$. Combining these two facts yields that $K_n^d$ is $(L+R_G(d)+R_H(d))$-dense in $X_n$. Recalling \eqref{toro}, we conclude that the $K_n^d$ form a \hamfam\ of $(G_n\square H_n)$ with control functions 
$$D_X=6 D_G +  6 D_H  \quad \text{ and } \quad 
R_X= 3D_G + 3D_H + R_G+R_H.$$ 
\end{proof}

\section{Boosting expanders} \label{sec exps}

Given a graph \g and $K \in \N$, we denote with \defi{$G^K$} the $K$-th \defi{power} of $G$, i.e.\ the graph on $V(G)$ with edge-set $\{xy \mid d_G(x,y)\leq K\}$.

The aim of this section is to prove that taking powers boosts the expansion constant $c$ of a \sCe\ (\Prr{exp boost} below). Combining this with \Tr{asham qi}, and the easy fact that $G^K$ is $(M(K),A(K))$-\qic\ to \G, proves the claim made in the introduction that to prove the first sentence of \Prb{prob exp} it suffices to prove it for large enough $C$.

\begin{proposition} \label{exp boost}
\Fe\ $1<C<C'$ there exists $K=K(C,C')\in\mathbb N$ such that if $G$ is an $n$-vertex $C$-expander, then $G^K$ is a $C'$-expander.
\end{proposition}

This follows immediately by recursively applying the following lemma (which shows that $K(C,C')= 3^{\ceil{\log_2(\log_C C')}}$ suffices): 

\begin{lemma} \label{cube boost}
Let $c>1$, and let $G$ be a $c$-expander. Then $G^3$ is a $c'$-expander \fe\ $c' 
\in [c, c^2]$.
\end{lemma}

\begin{proof}
For any $Y \subseteq V$ and \nin, let $S_n(Y):= \{v\in V \mid d_G(v,Y) = n\}$, and recall $B_n(Y):= \{v\in V \mid d_G(v,Y) \leq n\}$.

To prove that $G^3$ satisfies  \ref{ex i} with $C=c'$, it suffices to check that $|B_3(X) \sm X| \geq c'|X|$ for every $X \subseteq V$ with $0<|X| <  n/2c'$. Suppose this is not the case for some such $X$, i.e.\ letting $x:=|X|$ we have 
\labtequ{contr}{$|B_3(X)| < (c'+1)x$.}

We have $|S_1(X)|\geq c x$ by applying \ref{ex i} to \G. Combining this with \eqref{contr} yields
\labtequ{Stwo}{$|S_2(X)|<(c'-c)x$}
because $X \dot{\cup} S_1(X) \dot{\cup}S_2(X) \subseteq B_3(X)$. 

If $|B_1(X)| < n/2c$, then applying \ref{ex i} again with $X$ replaced by $B_1(X)$ we obtain $|S_2(X)| \geq c^2 x\geq c' x$, contradicting \eqref{Stwo}. 

Thus we can assume $|B_1(X)| \geq t:= n/2c$ from now on. Let $R:= V \sm B_3(X)$. Note that $R$ cannot be empty, because if it is, then we have $|B_3(X)| = |V|=n\geq 2c' x \geq c'x + x$, violating \eqref{contr}. 


Since there is no edge between $B_1(X)$ and $S_3\cup R$ (\fig{}),  \ref{ex ii} combined with $|B_1(X)| \geq t$ implies that $|S_3\cup R| <t$, and hence also $r:= |R| <t$. Thus applying \ref{ex i} to $R$ we have $|N(R)|\geq c r$. Since $N(R) \subseteq S_3$, we deduce that $|S_3\cup R| \geq (c+1) r$. 

Similarly, we can apply \ref{ex i} again to $S_3\cup R$ to obtain 
\labtequ{S low}{$|S_2|\geq c(c+1) r$,}
whereby we used that $|S_3\cup R| <t$, and that $N(S_3\cup R) \subseteq S_2$. On the other hand, combining \eqref{Stwo} with $c'\leq c^2$ yields $|S_2(X)|< c(c-1)x$, and combining this with \eqref{S low} we obtain
$$r< \frac{c-1}{c+1} x.$$
Since $B_3 \cup R=V$, it follows that
$|B_3| = n - r > 2c'x - \frac{c-1}{c+1} x$, which combined with \eqref{contr} yields
$$2c' - \frac{c-1}{c+1}  < c'+1.$$
Subtracting $c'$, and multiplying both sides with $c+1$ yields $c'(c+1) - c+1 <c+1$, and therefore $(c+1)c'<2c$, contradicting $c'\geq c\geq 1$. This contradiction proves \ref{ex i}.
\medskip

To prove \ref{ex ii}, pick disjoint sets $X_1,X_2 \subseteq V$, each with at least $n/2c'$ vertices. If each $X_i$ has at least $n/2c$ vertices we are done by applying \ref{ex ii} to $G$ with $C=c$, so let us define
\begin{equation}
 X'_i:=
 \begin{cases}
  X_i, & \text{if } |X_i|\geq n/2c,\\[2mm]
  N(X_i), & \text{if } |X_i|<n/2c.
 \end{cases}
\end{equation}

Note that $|X_i| < n/2c$ implies $|N(X_i)|\geq c |X_i|$ by applying \ref{ex i} to \G, and plugging in $|X_i| \geq n/2c'$ we deduce that $|X'_i|  \geq c n/2c' \geq c n/2c^2 = n/2c$ holds in both cases. 

If $X'_1,X'_2$ are disjoint, then applying \ref{ex ii} to \G\ we obtain an edge $e=x_1 x_2$ joining them, where each $x_i$ lies in $X'_i\subset X_i \cup N(X_i)$. It follows that $d_G(X_1,X_2)\leq 3$. If on the other hand $X'_1$ and $X'_2$ intersect, then easily $d_G(X_1,X_2)\leq 2$. Thus in both cases  $G^3$ contains an $X_1$--$X_2$~edge as desired. \end{proof}


\comment{
\subsection{purely graph theoretic problems}

For a function $f: \mathbb{Z}_{+} \to \mathbb{R}$, we say that a graph $G$ is \defi{$f$-connected} if
$|A \cap B| \geq f\left(\min\left\{|A\setminus B|,\;|B\setminus A|\right\}\right)$
for every two subsets $A,B \subseteq V(G)$ such that $V(G)=A\cup B$ and there is no edge between $A\setminus B$ and $B\setminus A$.

Brandt, Broersma, Diestel, and Kriesell \cite{BBDKGlo} conjectured  that there exists a linear function $f$ \st\ every $f$-connected graph with at least 3 vertices is Hamiltonian. This was proved by Dragani\'c, Montgomery, Munh\'a Correia, Pokrovskiy \& Sudakov \cite{DMMPSHam}.

Define $g: \R_+ \to \R_+$ by letting $g(R)$ be the infimal constant $B$ \st\ letting $f(k):= Bk$,  every $f$-connected graph on at least 3 vertices has an $R$-dense cycle. The aforementioned result of Dragani\'c et al.\ says that $g(R)$ is well-defined as $g(0)<\infty$. 
\begin{problem}  \label{prob gR}
Is $\lim_{R \to \infty} g(R) =0$?
\end{problem} 
}

\subsection*{Use of AI}

A mid-range AI agent made available by the University of Warwick, from now on abbreviated as CheapGPT,  was used to prove \Lr{cube boost}, to streamline the proofs of \Lr{lem fS D} and \Prr{kinfty}, and for other minor tasks. The main concepts, statements, and proofs of the paper are due to mankind, and I have thoroughly checked, and mostly re-written, any passages provided by CheapGPT. 

\acknowledgement{I thank Chiara Molinari for helpful discussions about the notion of \asc\ connectivity, in particular, for a proof of \Prr{aslykcon}. }

\bibliographystyle{plain}
\bibliography{collective}

\end{document}